\documentclass[10pt,a4paper,reqno]{amsart}
\usepackage{amsthm}
\usepackage{amsmath}
\usepackage{amssymb}
\usepackage{graphicx,color}

\usepackage[font=small]{caption}
\usepackage{subcaption}

\usepackage{multirow}
\usepackage[shortlabels]{enumitem}

\makeatletter 
\@namedef{subjclassname@2020}{\textup{2020} Mathematics Subject Classification}
\makeatother

\makeatletter
\theoremstyle{plain}
\newtheorem{thm}{Theorem}
  \theoremstyle{definition}
  \newtheorem*{thm*}{Theorem}
  
  \theoremstyle{remark}
  \newtheorem{rem}[thm]{Remark}
  \theoremstyle{plain}
  \newtheorem{prop}[thm]{Proposition}
  \theoremstyle{plain}
  \newtheorem{lem}[thm]{Lemma}
  \theoremstyle{plain}
  \newtheorem{cor}[thm]{Corollary}
 \theoremstyle{definition}
  
  \theoremstyle{remark}
  \newtheorem*{rem*}{Remark}

  \theoremstyle{definition}

\usepackage{amsfonts}
\usepackage{mathrsfs}

\newtheorem*{question*}{\it{QUESTION}}

\theoremstyle{plain}

\newcommand{\N}{\mathbb{N}}
\newcommand{\R}{{\mathbb{R}}}
\newcommand{\C}{{\mathbb{C}}}

\newcommand{\dd}{{\rm d}}
\newcommand{\ii}{{\rm i}}
\newcommand{\ee}{{\rm e}}

\renewcommand{\H}{{\rm H}}
\newcommand{\Arg}{{\rm Arg}}

\renewcommand{\Re}{\mathop\mathrm{Re}\nolimits}
\renewcommand{\Im}{\mathop\mathrm{Im}\nolimits}

\makeatother

\begin{document}

\title[]{A Herglotz--Nevanlinna function from the optimal discrete $p$-Hardy weight}

\author{Franti\v sek \v Stampach}
\address[Franti{\v s}ek {\v S}tampach]{
	Department of Mathematics, Faculty of Nuclear Sciences and Physical Engineering, Czech Technical University in Prague, Trojanova~13, 12000 Praha~2, Czech Republic
	}	
\email{stampfra@cvut.cz}

\author{Jakub Waclawek}
\address[Jakub Waclawek]{
	Department of Mathematics, Faculty of Nuclear Sciences and Physical Engineering, Czech Technical University in Prague, Trojanova~13, 12000 Praha~2, Czech Republic
	}	
\email{waclajak@cvut.cz}



\date{\today}

\begin{abstract}
It was recently proved by Fischer, Keller, and Pogorzelski in [Integr. Equ. Oper. Theory, 95(24), 2023] that the classical discrete $p$-Hardy inequality admits an improvement, and the optimal $p$-Hardy weight $\omega_{p}$ was determined therein.  We prove that $\omega_{p}$ directly corresponds to a Herglotz--Nevanlinna function, establish an integral representation for this function, and consequently confirm a~slight modification of a~conjecture on its absolute monotonicity from the aforementioned article.
\end{abstract}

\maketitle

\section{Introduction and main results}

Herglotz--Nevanlinna functions, also known as Pick functions, are analytic functions in the upper half-plane, mapping the upper half-plane into itself. This well-established class of complex functions has demonstrated its utility across a wide range of mathematical disciplines, such as the spectral theory of Jacobi and Schr{\" o}dinger operators~\cite{tes_00, tes_14}, moment problem and orthogonal polynomials~\cite{akh_21}, theory of operator monotone functions~\cite{sim_19}, canonical systems~\cite{rem_18}, de Branges' functional model spaces~\cite{debra_68}, theory of self-adjoint extensions~\cite{ber-has-sno_20}, and many more.  In this note, we highlight their relevance also in the context of \emph{Hardy inequalities}.

\subsection{The optimal discrete $p$-Hardy inequality}

In a 1921 letter to Hardy, Landau proved the inequality 
\begin{equation}
 \sum_{n=1}^{\infty}\left(\frac{a_1+\dots+a_n}{n}\right)^{p}\leq\left(\frac{p}{p-1}\right)^{p}\sum_{n=1}^{\infty}a_{n}^{p}
 \label{eq:hardy_ineq_class}
\end{equation}
for any $p>1$ and $a_n\geq0$, where the constant $(p/(p-1))^{p}$ is the best possible (Hardy proved the inequality with a larger constant). Many proofs of~\eqref{eq:hardy_ineq_class} are known today. A short one due to Elliot~\cite{ell_26} is given in the classical book~\cite[Theorem~326]{har-lit-pol_52}; see also \cite{lef_20} for an alternative short proof and~\cite{kuf-lec-per_06} for a historical survey. 

Inequality~\eqref{eq:hardy_ineq_class} can be rewritten to the difference form
\begin{equation}
\sum_{n=1}^{\infty}|\phi_{n}-\phi_{n-1}|^{p}\geq\sum_{n=1}^{\infty}\omega(n)|\phi_{n}|^{p},
\label{eq:hardy_ineq}
\end{equation}
with the classical $p$-Hardy weight
\[
\omega(n)=\omega_{p}^{\H}(n):=\left(\frac{p-1}{p}\right)^{p}\frac{1}{n^{p}}
\]
for any compactly supported sequence $\phi\in C_{\mathrm{c}}(\N)$ with $\phi_0:=0$. Despite the optimality of the multiplicative constant, $\omega_{p}^{\H}$ admits an improvement. This interesting fact has been observed only recently by Fischer, Keller, and Pogorzelski, who proved in~\cite{fis-kel-pog_23} that the $p$-Hardy inequality~\eqref{eq:hardy_ineq} holds with the weight $\omega$ given by
\begin{equation}
 \omega_{p}(n):=\left(1-\left(1-\frac{1}{n}\right)^{1/q}\right)^{p-1}-\left(\left(1+\frac{1}{n}\right)^{1/q}-1\right)^{p-1},
\label{eq:def_om_p}
\end{equation}
where $q>1$ is the H{\" o}lder conjugate index to $p$, i.e. $1/p+1/q=1$. Moreover, the weight $\omega_{p}$ has been shown to be \emph{optimal} in a strong sense. In particular, the optimality involves the \emph{criticality} of $\omega_{p}$, which means that,  if~\eqref{eq:hardy_ineq} holds with a weight $\omega(n)\geq\omega_{p}(n)$ for all $n\in\N$ and all $\phi\in C_{\mathrm{c}}(\N)$, then necessarily $\omega=\omega_{p}$. In other words, the $p$-Hardy inequality cannot hold for all $\phi\in C_{\mathrm{c}}(\N)$ with a weight pointwise greater than $\omega_{p}$.
The proof of inequality~\eqref{eq:hardy_ineq} with $\omega=\omega_{p}$ as well as the optimality of $\omega_p$ presented in~\cite{fis-kel-pog_23} relies on methods developed in greater generality for graphs, see dissertation thesis~\cite{fis_diss24}, and~\cite{fis_24}. 

A substantial part of~\cite{fis-kel-pog_23} is actually devoted to demonstrating that the weight~\eqref{eq:def_om_p} improves upon $\omega_{p}^{\H}$, i.e. to the nontrivial inequality
\begin{equation}
    \omega_{p}(n)>\omega_{p}^{\H}(n), \quad \forall n\in\N, 
    \label{eq:om_p>om_p^H}
\end{equation}
which is established for all $p>1$. If, in addition, $p$ is an \textbf{integer}, it is further shown in~\cite{fis-kel-pog_23} the expansion
\begin{equation}
    \omega_{p}(n)=\sum_{k=0}^{\infty}\frac{m_{2k}^{(p)}}{n^{2k+p}}, \quad \forall n\in\N, 
    \label{eq:om_expan_integer}
\end{equation}
where the coefficients $m_{2k}^{(p)}$ are \textbf{strictly positive} for all $k\in\N_{0}$. Since $m_{0}^{(p)}=((p-1)/p)^{p}$, inequality~\eqref{eq:om_p>om_p^H} is an immediate consequence. This is a noteworthy property of the optimal $p$-Hardy weight $\omega_p$, one that does not follow from its optimality. The reason is that any number of terms from the expansion \eqref{eq:om_expan_integer} can be used to produce a $p$-Hardy weight into~\eqref{eq:hardy_ineq}, with the resulting weight becoming tighter as more terms are included. 

However, it remained an open problem whether~\eqref{eq:om_expan_integer} also holds for non-integer $p>1$. We prove that this is indeed the case. In fact, it was conjectured in~\cite{fis-kel-pog_23} that the function $x\mapsto \omega_{p}(1/x)$ is \emph{absolutely monotonic} on $[0,1]$; a property that implies non-negativity of the coefficients in~\eqref{eq:om_expan_integer}. This conjecture is true only if $p\in\N$. Nevertheless, we prove that the conjecture holds for every $p>1$ if the function $x\mapsto x^{-p}\omega_{p}(1/x)$ is considered instead, see Corollary~\ref{cor:abs-mon} below.

Discrete Hardy inequalities have received greater attention in the more tractable particular case $p=2$. The study of optimal Hardy inequalities on graphs was initiated in~\cite{kel-pin-pog_18}, motivated by theory of Hardy inequalities for PDEs~\cite{dev-fra-pin_14}. Simple proofs of the optimal discrete Hardy inequality for $p=2$ can be found in~\cite{kel-pin-pog_18a, kre-sta_22}, see also~\cite[Sec.~4]{kre-lap-sta_22}. Discrete Hardy inequalities of higher order, in particular discrete Rellich inequalities, have been investigated in~\cite{ger-kre-sta_25a, hua-ye_24}. Optimal discrete Hardy inequalities of arbitrary order were discovered in~\cite{sta-wac_prep}. Further recent works contributing to the field of discrete Hardy-type inequalities include~\cite{cia-ron_18, das-etal_25, ger-kre-sta_25b, gup_22, gup_24, kel-nit_23, pra-dur_24}.

The paper is organised as follows. Section~\ref{subsec:main} presents our main results, which are subsequently proven in Section~\ref{sec:proofs}. A direct and elementary proof of the positive coefficient expansion~\eqref{eq:om_expan_integer} for integer values of $p$ is provided in Section~\ref{subsec:2.1}, bypassing the theory of Herglotz--Nevanlinna functions. 
Section~\ref{subsec:2.2} reviews needed results from complex analysis. The main result and its corollary, which confirms the corrected conjecture from~\cite{fis-kel-pog_23}, are proven in Sections~\ref{subsec:2.3} and~\ref{subsec:2.4}, respectively. 
As a supplementary result, Section~\ref{sec:comb-form} derives a~formula for the coefficients in~\eqref{eq:om_expan_integer}.

\subsection{Main result}\label{subsec:main}

Before formulating our main results, we introduce a notation.
Considering $\omega_{p}$, determined by the expression~\eqref{eq:def_om_p}, as a function of complex variable, we define function $f_{p}(z):=-z^{p-1}\omega_{p}(z)$, i.e.
\begin{equation}
 f_{p}(z)=z^{p-1}\left[\left(\left(1+\frac{1}{z}\right)^{1/q}-1\right)^{p-1}-\left(1-\left(1-\frac{1}{z}\right)^{1/q}\right)^{p-1}\right]
\label{eq:def_f_p}
\end{equation}
for $z\in\C\setminus\R$ and $p>1$, where $1/p+1/q=1$.
Here and below, the positive power of a~complex argument assumes its principle branch, i.e. for $\alpha>0$, we always make use of the definition
\begin{equation}
 z^{\alpha}:=|z|^{\alpha}\ee^{\ii\alpha\Arg\, z}
\label{eq:pow_func}
\end{equation}
with $\Arg\, z\in(-\pi,\pi]$. It turns out that $f_{p}$ extends analytically to $\C\setminus[-1,1]$, with a branch cut on $[-1,1]$ (see Lemma~\ref{lem:lem1}).

Further, for $x\in[0,1]$, we define function
\begin{equation}
\rho_{p}(x):=-\frac{1}{\pi}\,x^{p-1}\Im\left(1-\ee^{\ii\pi/q}\left(\frac{1}{x}-1\right)^{1/q}\right)^{p-1}.
\label{eq:rho_p}
\end{equation}
The value of $\rho_{p}$ at $0$ is to be understood as the limit of $\rho_{p}(x)$ for $x\to0_{+}$. We will show that $\rho_{p}(0)=0$ for any $p>1$ (see Lemma~\ref{lem:lem2}). The function $\rho_{p}$ turns out to be the density of a measure whose even moments coincide with the coefficients from the expansion~\eqref{eq:om_expan_integer}; see Corollary~\ref{cor:abs-mon} below.

We denote the upper and the lower half-plane by 
\[
\C_{\pm}:=\{z\in\C \mid \pm\Im z>0\}.
\]
Recall that a function $f:\C_{+}\to\C$ is called \emph{Herglotz--Nevanlinna}, if $f$ is analytic on $\C_{+}$ and $\Im f(z)\geq0$ for all $z\in\C_{+}$. Our main result reads:

\begin{thm}\label{thm:main}
Let $p>1$. The function $f_{p}$ defined by~\eqref{eq:def_f_p} is Herglotz--Nevanlinna. Moreover, the density $\rho_{p}$ defined by~\eqref{eq:rho_p} is strictly positive on $(0,1)$ and the integral representation
\begin{equation}
f_{p}(z)=\int_{-1}^{1}\frac{\rho_{p}(|t|)}{t-z}\dd t
\label{eq:int-repre}
\end{equation}
holds for all $z\in\C\setminus[-1,1]$.
\end{thm}

 Recall that a real-valued function $g$ is called \emph{absolutely monotonic} in $[0,1]$, if it is continuous on $[0,1]$, infinitely-differentiable on $(0,1)$, and $g^{(n)}(x)\geq0$ for all $x\in(0,1)$ and all $n\in\N_{0}$. 
 For a comprehensive introduction to absolutely monotonic functions, we refer the reader to~\cite[Chap.~IV]{wid_41}. 
 The integral representation~\eqref{eq:int-repre} implies a property of $\omega_p$ that proves a~corrected form of a conjecture from~\cite{fis-kel-pog_23}:
 
\begin{cor}\label{cor:abs-mon}
For any $p>1$, the function 
\begin{equation}
x\mapsto\frac{1}{x^{p}}\,\omega_{p}\!\left(\frac{1}{x}\right)
\label{eq:func}
\end{equation}
is absolutely monotonic on $[0,1]$. Moreover, for all $x\in(0,1]$, we have the convergent expansion
\begin{equation}
    \frac{1}{x^{p}}\omega_{p}\left(\frac{1}{x}\right)=\sum_{k=0}^{\infty}m_{2k}^{(p)}x^{2k}
\label{eq:om_expan}
\end{equation}
with coefficients
\begin{equation}
m_{2k}^{(p)}:=\int_{-1}^{1}t^{2k}\rho_{p}(|t|)\dd t>0
\label{eq:def_mom}
\end{equation}
for all $k\in\N_{0}$.
\end{cor}

\begin{rem}
In fact, the derivatives of~\eqref{eq:func} of all orders are strictly positive in $(0,1)$.
\end{rem}

As an immediate consequence of Corollary~\ref{cor:abs-mon}, we get the expansion~\eqref{eq:om_expan_integer} for all~$p>1$, with $m_{2k}^{(p)}>0$ for all $k\in\N_{0}$. In Proposition~\ref{prop:comb_formula} below, we provide an explicit but combinatorially complicated formula for the coefficients $m_{2k}^{(p)}$ for all $k\in\N_{0}$ and $p>1$. For $p$ integer, the formula simplifies; see Remark~\ref{rem:mom_p_int} below.

\section{Proofs}\label{sec:proofs}

\subsection{A warm up: the case of integer $p$}\label{subsec:2.1}

In the context of the discrete $p$-Hardy inequalities, key properties of the function $f_{p}$ are the integral formula~\eqref{eq:int-repre}, positivity of $\rho_{p}$ on $(0,1)$, and finite (actually vanishing) side limits $\rho_{p}(0)=\rho_{p}(1)=0$ since they directly lead to the expansion~\eqref{eq:om_expan_integer} with positive coefficients. 
Notably, when $p>1$ is an integer, these properties can be established using elementary methods, bypassing the need for Herglotz--Nevanlinna theory. As an exercise and an alternative demonstration of the results from~\cite{fis-kel-pog_23}, we offer a concise direct proof. Readers primarily interested in the proof of Theorem~\ref{thm:main} for general $p>1$ may safely skip this subsection.

Assuming $p\in\N$, $p>1$, we first check the positivity of $\rho_{p}$ on $(0,1)$. Using the binomial formula, we may rewrite the definition \eqref{eq:rho_p} as
\begin{equation} \label{eq:rho_p-alt}
    \rho_p(x)=-\frac{1}{\pi} \sum_{j=0}^{p-1} \binom{p-1}{j} (-1)^j \sin{\left(\frac{j\pi}{q}\right)} (1-x)^{j/q}x^{p-1-j/q}.
\end{equation}
From here, it is immediate that $\rho_{p}(x)>0$ for all $x\in(0,1)$ since $(-1)^{j}\sin(j\pi/q)<0$ for all $1\leq j\leq p-1$. Moreover, the side limit values $\rho_{p}(0)=\rho_{p}(1)=0$ also follow.

While it is obvious that the integral in~\eqref{eq:int-repre} is analytic in $\C\setminus[-1,1]$, one can verify the same for $f_{p}$ by inspection of the defining expression~\eqref{eq:def_f_p}; this is shown in Lemma \ref{lem:lem1} below in detail. Therefore it suffices to verify~\eqref{eq:int-repre} for $z\in\C$ in a neighborhood of $\infty$. Equation~\eqref{eq:int-repre} rewrites equivalently as
\begin{equation} \label{eq:int-repre-integer}
    -\frac{1}{z}f_{p}\left(\frac{1}{z}\right) = \int_{-1}^{1}\frac{\rho_{p}(|t|)}{1-tz}\dd t = 2\int_{0}^{1}\frac{\rho_{p}(t)}{1-t^2z^2}\dd t.
\end{equation}
We prove~\eqref{eq:int-repre-integer} by expanding both sides into power series in $z$, assuming $|z|<1$, and showing coincidence of the corresponding coefficients.

\emph{a) Expansion of the left-hand side of~\eqref{eq:int-repre-integer}:} By applying the (generalized) binomial theorem in~\eqref{eq:def_f_p} twice, we find
\begin{align*}
    -\frac{1}{z}f_p\left(\frac{1}{z}\right) &= \frac{1}{z^p} \sum_{j=0}^{p-1} \binom{p-1}{j} (-1)^j \left( (1-z)^{j/q} - (-1)^{p-1} (1+z)^{j/q} \right) \\
        &= \frac{1}{z^p} \sum_{j=0}^{p-1} \binom{p-1}{j} (-1)^j \sum_{k=0}^\infty \binom{j/q}{k} \left( (-z)^k - (-1)^{p-1}z^k \right)
\end{align*}
for $z\in\C$ with $|z|<1$. Thus, we arrive at the formula
\begin{equation} \label{eq:LHS-exp-cases}
    -\frac{1}{z}f_p\left(\frac{1}{z}\right) =
    \begin{cases}
        2\displaystyle \sum\limits_{n=0}^\infty z^{2n-p} \sum\limits_{j=0}^{p-1} (-1)^j \binom{p-1}{j}\binom{j/q}{2n} & \mbox{ for } p\in 2\N, \\[4pt]
        2\displaystyle \sum\limits_{n=0}^\infty z^{2n+1-p} \sum\limits_{j=0}^{p-1} (-1)^{j+1} \binom{p-1}{j}\binom{j/q}{2n+1} & \mbox{ for } p\in 2\N-1.
    \end{cases}
\end{equation}
Notice that for any $\ell\in\N$, $\alpha>0$, and $m\in\{0,1,\dots,\ell-1\}$, we have
\[
        \sum_{j=0}^\ell \binom{\ell}{j}\binom{\alpha j}{m} (-1)^j = 0,
\]
which is a consequence of the identity~\cite[Eq.~0.154.3]{grad-ryz_07}
   \[
        \sum_{j=0}^\ell \binom{\ell}{j} j^m (-1)^j = 0, \quad\forall m\in\{0,1,\dots,\ell-1\}.
   \]
As a result, the terms in~\eqref{eq:LHS-exp-cases} corresponding to $n=0,1,\dots,\lfloor(p-2)/2\rfloor$ vanish. Using this observation and shifting the index $n$, we obtain the expansion
\begin{equation} \label{eq:LHS-expansion}
-\frac{1}{z}f_p\left(\frac{1}{z}\right) = 2 \sum\limits_{n=0}^\infty z^{2n} \sum\limits_{j=0}^{p-1} (-1)^{j+p} \binom{p-1}{j}\binom{j/q}{2n+p}, \quad |z|<1.    
\end{equation}

\emph{b) Expansion of the right-hand side of~\eqref{eq:int-repre-integer}:}  For $z\in\C$ with $|z|<1$, we have 
\[
    \int_{0}^{1}\frac{\rho_{p}(t)}{1-t^2z^2}\dd t = \sum_{n=0}^\infty z^{2n} \int_0^1 t^{2n}\rho_p(t) \dd t.
\]
Bearing \eqref{eq:rho_p-alt} in mind, we further get
\[
    \int_{0}^{1}\frac{\rho_{p}(t)}{1-t^2z^2}\dd t = \sum_{n=0}^\infty z^{2n} \left(\frac{1}{\pi} \sum_{j=0}^{p-1} \binom{p-1}{j} (-1)^{j+1} \sin{\left(\frac{j\pi}{q}\right)} \int_0^1 (1-t)^{j/q}t^{p-1-j/q+2n} \dd t \ \right)\!.
\]
The latter integral can be computed using the Euler beta and gamma functions as
\[
    \int_0^1 (1-t)^{j/q}t^{p-1-j/q+2n} \dd t = \mathrm{B} \left(2n+p-\frac{j}{q},1+\frac{j}{q}\right) = \frac{\Gamma\left(2n+p-\frac{j}{q}\right)\Gamma\left(1+\frac{j}{q}\right)}{\Gamma\left(2n+p+1\right)},
\]
which further simplifies to
\[
    \int_0^1 (1-t)^{j/q}t^{p-1-j/q+2n} \dd t=(-1)^{p+1}\binom{j/q}{2n+p} \frac{\pi}{\sin(j\pi/q)}
\]
by employing the identity $\Gamma(x+1)=x\Gamma(x)$ and Euler's reflection formula. Altogether, we deduce for the right-hand side of~\eqref{eq:int-repre-integer} the same expansion as in~\eqref{eq:LHS-expansion}. The proof of~\eqref{eq:int-repre} for integer $p$ follows.

\begin{rem}
The crucial simplifications inherent in the preceding method are not applicable for non-integer values of $p>1$; cf. Proposition~\ref{prop:comb_formula} below. This discrepancy suggests that establishing~\eqref{eq:int-repre} for non-integer $p>1$ requires a different strategy.
\end{rem}

\subsection{Two theorems from complex analysis}\label{subsec:2.2}

In our proof of Theorem~\ref{thm:main}, we will make use of two facts from complex analysis. 
The first one is the following property of bounded analytic functions on $\C_{+}$ which follows from the 
Phragmén--Lindel{\" o}f principal, see for example~\cite[Sec.~III.C]{koo_98}.

\begin{thm}\label{thm:p-l}
Let $f$ be a bounded analytic function on $\C_{+}$ continuous up to the real line~$\R$. If $\Im f(x)\geq0$ for all $x\in\R$, then $\Im f(z)\geq0$ for all $z\in\C_{+}$.
\end{thm}


The second theorem compiles selected properties of Herglotz--Nevanlinna functions, notably their integral representation. These results are standard; proofs can be found for instance in~\cite[\S3.4]{tes_14}. Recall that the domain of any Herglotz--Nevanlinna function $f$ can be extended from the upper half-plane $\C_{+}$ to the cut-plane $\C\setminus\R$ by the Schwarz reflection $f(\overline{z})=\overline{f(z)}$. This extension is assumed hereafter whenever a Herglotz-Nevanlinna function is considered on $\C\setminus\R$.

\begin{thm}\label{thm:h-n}
Let $f:\C\setminus\R\to\C$ be a Herglotz--Nevanlinna function. Then the following claims hold true:
\begin{enumerate}[i)]
\item There exist unique constants $a\in\R$, $b\geq0$, and a unique Borel measure $\mu$ on $\R$ satisfying 
\[
 \int_{\R}\frac{\dd\mu(t)}{1+t^{2}}<\infty
\]
such that
\[
 f(z)=a+bz+\int_{\R}\left(\frac{1}{t-z}-\frac{t}{1+t^{2}}\right)\dd\mu(t)
\]
for all $z\in\C\setminus\R$.
\item If the limit of $\Im f(x+\ii\varepsilon)$ for $\varepsilon\to0+$ exists and is finite for all $x\in\R$, then the measure $\mu$ is purely absolutely continuous and its density reads
\[
 \frac{\dd\mu}{\dd x}(x)=\frac{1}{\pi}\lim_{\varepsilon\to0+}\Im f(x+\ii\varepsilon)
\]
for $x\in\R$.
\item If $\mu$ is a compactly supported finite measure, then 
\[
 f(z)=bz+a-\int_{\R}\frac{t\dd\mu(t)}{1+t^{2}}-\frac{\mu(\R)}{z}+\mathcal{O}\left(\frac{1}{z^{2}}\right) \quad\mbox{ as } z\to\infty.
\]
\end{enumerate}
\end{thm}

\subsection{Proof of Theorem~\ref{thm:main}}\label{subsec:2.3}

First, we deduce preliminary properties of the function $f_{p}$ defined by~\eqref{eq:def_f_p}. We assume $p>1$ everywhere.

\begin{lem}\label{lem:lem1}
 Function $f_{p}$ extends analytically to the cut-plane $\C\setminus[-1,1]$, satisfies the symmetry relations
\begin{equation}
 \overline{f_{p}(z)}=f_{p}(\overline{z}) 
 \quad\mbox{ and }\quad 
 f_{p}(-z)=-f_{p}(z)
\label{eq:f_p_symm}
\end{equation}
for all $z\in\C\setminus\R$, 
and the asymptotic expansion 
\begin{equation}
f_{p}(z)=-\left(\frac{p-1}{p}\right)^{p}\frac{1}{z}+\mathcal{O}\left(\frac{1}{z^{2}}\right) \quad\mbox{ as }z\to\infty.
\label{eq:f_p_z_large}
\end{equation}
\end{lem}

\begin{proof}
First, we verify the symmetry formulas~\eqref{eq:f_p_symm}. 
Since the principal branch of the power function~\eqref{eq:pow_func} enjoys the symmetry $\overline{z^{\alpha}}=\overline{z}^{\,\alpha}$ for all $z\in\C\setminus(-\infty,0]$ the first symmetry formula from~\eqref{eq:f_p_symm} for $f_{p}$ follows readily from its definition~\eqref{eq:def_f_p}. To prove the second symmetry from~\eqref{eq:f_p_symm}, first observe that, for any $\alpha>0$, we have
\[
 (-z)^{\alpha}=\ee^{\mp\ii\alpha\pi}z^{\alpha} \quad\mbox{ if } z\in\C_{\pm}.
\]
It suffices to take $z\in\C_{+}$, for which we have
\begin{equation}
 \pm1\mp\left(1\mp\frac{1}{z}\right)^{1/q}\in\C_{-},
\label{eq:aux_rel}
\end{equation}
here it is important that $q>1$. Now, we find
\begin{align*}
 f_{p}(-z)&=(-z)^{p-1}\left[\left(\left(1-\frac{1}{z}\right)^{1/q}-1\right)^{p-1}-\left(1-\left(1+\frac{1}{z}\right)^{1/q}\right)^{p-1}\right]\\
 &=\ee^{-\ii(p-1)\pi}z^{p-1}
 \left[\ee^{\ii(p-1)\pi}\left(1-\left(1-\frac{1}{z}\right)^{1/q}\right)^{p-1}-\ee^{\ii(p-1)\pi}\left(\left(1+\frac{1}{z}\right)^{1/q}-1\right)^{p-1}\right]\\
 &=-f_{p}(z)
\end{align*}
for any $z\in\C_{+}$, which implies that $f_{p}(-z)=-f_{p}(z)$ for all $z\in\C\setminus\R$.

Second, we prove the analyticity of $f_{p}$ in $\C\setminus[-1,1]$.
Since the function $z\mapsto z^{\alpha}$ defined by~\eqref{eq:pow_func} is analytic in $\C\setminus(-\infty,0]$, the analyticity of $f_{p}$ in $\C\setminus\R$ is immediate from its definition~\eqref{eq:def_f_p}. To show that $f_{p}$ is analytic on $\C\setminus[-1,1]$, it suffices to show that $f_{p}$ extends continuously to $\R\setminus[-1,1]$. For any $x>1$, one computes, by inspection of the expression~\eqref{eq:def_f_p}, that
\begin{equation}
 \lim_{\substack{z\to x \\ z\in\C\setminus\R}} f_{p}(z)=
 x^{p-1}\left[\left(\left(1+\frac{1}{x}\right)^{1/q}-1\right)^{p-1}-\left(1-\left(1-\frac{1}{x}\right)^{1/q}\right)^{p-1}\right].
\label{eq:f_p_bv_x>1}
\end{equation}
The continuity of $f_{p}$ on $(1,\infty)$ follows. The continuity of $f_{p}$ on $(-\infty,-1)$ is a consequence of the already proven formulas~\eqref{eq:f_p_symm}.

Lastly, we derive the asymptotic expansion~\eqref{eq:f_p_z_large}.
Due to symmetries~\eqref{eq:f_p_symm}, we may confine $z$ to $\C_{+}\cup(1,\infty)$. Since $(zw)^{\alpha}=z^{\alpha}w^{\alpha}$ for any $z\in\C_{+}\cup(1,\infty)$ and $w\in\C_{-}\cup(1,\infty)$, we may rewrite~\eqref{eq:def_f_p}, bearing~\eqref{eq:aux_rel} in mind, as
\begin{equation}
f_{p}(z)=\left(z\left(1+\frac{1}{z}\right)^{1/q}-z\right)^{p-1}-\left(z-z\left(1-\frac{1}{z}\right)^{1/q}\right)^{p-1}
\label{eq:f_p_alt_expr}
\end{equation}
for all $z\in\C_{+}\cup(1,\infty)$. Now, using the generalized binomial theorem twice, one readily expands
\begin{align*}
f_{p}(z)&=\frac{1}{q^{p-1}}\left[\left(1+\left(\frac{1}{q}-1\right)\frac{1}{2z}+\mathcal{O}\left(\frac{1}{z^{2}}\right)\right)^{p-1}-\left(1-\left(\frac{1}{q}-1\right)\frac{1}{2z}+\mathcal{O}\left(\frac{1}{z^{2}}\right)\right)^{p-1}\right]\\
&=\frac{1}{q^{p-1}}\left[\left(\frac{1}{q}-1\right)\frac{p-1}{z}+\mathcal{O}\left(\frac{1}{z^{2}}\right)\right]
\end{align*}
for $z\to\infty$. Recalling that $1/p+1/q=1$, the last formula simplifies to~\eqref{eq:f_p_z_large}.
\end{proof}

In the second lemma, we explore the continuation of $f_{p}$ onto the cut $[-1,1]$ from the upper half-plane~$\C_{+}$.

\begin{lem}\label{lem:lem2}
 The function $f_{p}$ extends continuously to the interval $[-1,1]$ from~$\C_{+}$. Moreover, for $x\in\R$, we have
 \begin{equation}
  \lim_{\substack{z\to x \\ z\in\C_{+}}}\Im f_{p}(z)=
  \begin{cases}  
  \pi\rho_{p}(|x|), &\quad\mbox{ if } |x|<1, \\
  0, &\quad\mbox{ if } |x|\geq1,
  \end{cases}
 \label{eq:f_p_bv_x>0}
 \end{equation}
 where $\rho_{p}$ is defined by~\eqref{eq:rho_p}; here the value for $x=0$ is given by the respective limit $\rho_{p}(0)=0$.
\end{lem}

\begin{proof}
By~\eqref{eq:f_p_symm}, it suffices to consider $x\geq0$.
For $x\in(0,1]$, one computes from~\eqref{eq:def_f_p} and definition~\eqref{eq:pow_func} that
\begin{equation}
\lim_{\substack{z\to x \\ z\in\C_{+}}}f_{p}(z)=
 x^{p-1}\left[\left(\left(1+\frac{1}{x}\right)^{1/q}-1\right)^{p-1}-\left(1-\ee^{\ii\pi/q}\left(\frac{1}{x}-1\right)^{1/q}\right)^{p-1}\right].
\label{eq:f_p_bv_0<x<1}
\end{equation}
For $x=0$, the limit is $0$, as one can see from the  expression~\eqref{eq:f_p_alt_expr}, where both summands vanish for $z\to0$ since
\[
 z\left(1\pm\frac{1}{z}\right)^{1/q}=\mathcal{O}(z^{1/p}) \quad\mbox{ as } z\to0.
\]
Consequently, $f_{p}$ extends continuously to the interval $[-1,1]$ from~$\C_{+}$.

Since the first term on the right-hand side of~\eqref{eq:f_p_bv_0<x<1} is real, we have
\[
\lim_{\substack{z\to x \\ z\in\C_{+}}}\Im f_{p}(z)=
 -x^{p-1}\Im\left(1-\ee^{\ii\pi/q}\left(\frac{1}{x}-1\right)^{1/q}\right)^{p-1},
\]
which coincides with $\pi\rho_{p}(x)$, for $x\in(0,1]$, see~\eqref{eq:rho_p}. If $x>1$, we already know from~\eqref{eq:f_p_bv_x>1} that the limit of $f_{p}(z)$, for $z\to x$, exists and is real.
This proves~\eqref{eq:f_p_bv_x>0}.
\end{proof}

\begin{lem}\label{lem:lem3}
    For all $x\in(0,1)$, $\rho_{p}(x)>0$.
\end{lem}

\begin{proof}
By definition~\eqref{eq:rho_p}, it suffices to verify that 
\[
\Im\left(1-\ee^{\ii\pi/q}\left(\frac{1}{x}-1\right)^{1/q}\right)^{p-1}<0
\]
for all $x\in(0,1)$. Denoting $t:=(1/x-1)^{1/q}$ and using~\eqref{eq:pow_func}, this inequality is true if
\begin{equation}
\sin\left((p-1)\Arg\left(1-\ee^{\ii\pi/q}t\right)\right)<0
\label{eq:sin_ineq}
\end{equation}
for all $t>0$. We verify~\eqref{eq:sin_ineq} separately for two cases when $p\in(1,2)$ and $p\geq2$.

Suppose first that $p\geq2$. Then $\pi/2\leq\pi/q<\pi$, $\Re(1-\ee^{\ii\pi/q}t)>0$, and we find
\[
0>\Arg\left(1-\ee^{\ii\pi/q}t\right)=\arctan\left(\frac{-t\sin(\pi/q)}{1-t\cos(\pi/q)}\right)>\arctan\left(\tan\left(\frac{\pi}{q}\right)\right)=\frac{\pi}{q}-\pi=-\frac{\pi}{p}.
\]
Therefore the argument of the sine in~\eqref{eq:sin_ineq} satisfies
\[
0>(p-1)\Arg\left(1-\ee^{\ii\pi/q}t\right)>\frac{1-p}{p}\pi=-\frac{\pi}{q}>-\pi
\]
and so inequality~\eqref{eq:sin_ineq} is true.

Second, suppose $p\in(1,2)$. Then $0<\pi/q<\pi/2$ and $\Re(1-\ee^{\ii\pi/q}t)=1-t\cos(\pi/q)$ changes sign as function of $t>0$. Let us denote by $t_{0}$ the unique positive point for which $\Re(1-\ee^{\ii\pi/q}t_0)=0$. If $t<t_{0}$, then $\Re(1-\ee^{\ii\pi/q}t)>0$, and we have
\[
-\frac{\pi}{2}<\Arg\left(1-\ee^{\ii\pi/q}t\right)=\arctan\left(\frac{-t\sin(\pi/q)}{1-t\cos(\pi/q)}\right)<0,
\]
because the argument of the inverse tangent is negative. If $t>t_{0}$, then $\Re(1-\ee^{\ii\pi/q}t)<0$, and we have
\[
-\frac{\pi}{2}>\Arg\left(1-\ee^{\ii\pi/q}t\right)=-\pi+\arctan\left(\frac{-t\sin(\pi/q)}{1-t\cos(\pi/q)}\right)>-\pi,
\]
because the argument of the inverse tangent is positive. If $t=t_0$, 
$\Arg(1-\ee^{\ii\pi/q}t)=-\pi/2$ because $\Im\Arg(1-\ee^{\ii\pi/q}t)=-t\sin(\pi/q)<0$. In total, we see that 
\[
-\pi<-(p-1)\pi<(p-1)\Arg\left(1-\ee^{\ii\pi/q}\right)<0,
\]
where the assumption $p\in(1,2)$ was used. Therefore the inequality~\eqref{eq:sin_ineq} remains true. 
\end{proof}

\begin{proof}[Proof of Theorem~\ref{thm:main}]

First, we prove that $f_{p}$ is Herglotz--Nevanlinna.
By Lemma~\ref{lem:lem1}, function $f_{p}$ is analytic on $\C_{+}$. Therefore $f_{p}$ is Herglotz--Nevanlinna if $\Im f_{p}(z)\geq0$ for all $z\in\C_{+}$. According to Lemmas~\ref{lem:lem1} and~\ref{lem:lem2}, $f_{p}$ extends continuously to~$\R$ from~$\C_{+}$ and is bounded on $\C_{+}$, the boundedness of $f_{p}$ by $\infty$ is a~consequence of~\eqref{eq:f_p_z_large}. Thus, utilizing Theorem~\ref{thm:p-l}, it suffices to show that the limit values of $\Im f_{p}(z)$, as $z\to x$ from $\C_{+}$, are nonnegative for all $x\in\R$. This is justified by Lemmas~\ref{lem:lem2} and~\ref{lem:lem3}.

As the positivity of $\rho_p$ in $(0,1)$ is established by Lemma~\ref{lem:lem3}, it remains to prove~\eqref{eq:int-repre}.
Since $f_{p}$ is Herglotz--Nevanlinna Theorem~\ref{thm:h-n} applies, providing us with numbers $a_p\in\R$, $b_p\geq0$, and measure $\mu_p$ so that the integral representation of claim~(i) holds. By claim~(ii) of Theorem~\ref{thm:h-n} and Lemmas~\ref{lem:lem2}, the measure $\mu_p$ is supported in $[-1,1]$, purely absolutely continuous, and its density reads
\[
 \frac{\dd\mu_p}{\dd x}(x)=\rho_{p}(|x|)
\]
for $x\in(-1,1)$. As the density $\rho_{p}$ is continuous on $[0,1]$, $\mu$ is a finite measure. Lastly, the claim~(iii) of Theorem~\ref{thm:h-n} combined with the expansion~\eqref{eq:f_p_z_large} implies
\[
 b_p=0, \quad a_p=\int_{-1}^{1}\frac{t\dd\mu_p(t)}{1+t^{2}}, \quad\mbox{ and moreover }\; \mu_p(\R)=\left(\frac{p-1}{p}\right)^{p}.
\]
In total, we arrive at the integral representation of the form~\eqref{eq:int-repre}, as claimed. The proof of Theorem~\ref{thm:main} is complete.
\end{proof}

\subsection{Proof of Corollary~\ref{cor:abs-mon}}\label{subsec:2.4}

One infers from~\eqref{eq:int-repre} the integral formula
\begin{equation}
 \frac{1}{x^{p}}\omega_{p}\left(\frac{1}{x}\right)=-\frac{1}{x}f_{p}\left(\frac{1}{x}\right)=\int_{-1}^{1}\frac{\rho_{p}(|t|)}{1-tx}\dd t
\label{eq:int_repre_alt}
\end{equation}
for all $x\in(-1,1)$. This is an infinitely-differentiable function of $x$ on $(-1,1)$ and, by the dominated convergence, its $n$th derivative satisfies
\[
\frac{\dd^{n}}{\dd x^{n}}\left[\frac{1}{x^{p}}\omega_{p}\left(\frac{1}{x}\right)\right]=n!\int_{-1}^{1}\frac{t^{n}\rho_{p}(|t|)}{(1-tx)^{n+1}}\dd t=n!\int_{0}^{1}\left(\frac{1}{(1-tx)^{n+1}}+\frac{(-1)^{n}}{(1+tx)^{n+1}}\right)t^{n}\rho_{p}(t)\dd t
\]
for all $x\in(-1,1)$. Since the expression in the round brackets in the last integral is positive for all $t,x\in(0,1)$ and the function in~\eqref{eq:int_repre_alt} is continuous on $[0,1]$, it is absolutely monotonic on $[0,1]$.

Expanding the integrand on the right-hand side of \eqref{eq:int_repre_alt} into the power series in the variable $xt$ and interchanging the sum and the integral yields
\[
\frac{1}{x^{p}}\omega_{p}\left(\frac{1}{x}\right)=\sum_{n=0}^{\infty}x^{n}\int_{-1}^{1}t^{n}\rho_{p}(|t|)\dd t 
\]
where the series on the right converges for all $x\in(-1,1)$. Since the last integral vanishes for $n$ odd, this yields the claimed expansion~\eqref{eq:om_expan} for $x\in(0,1)$. The extension of~\eqref{eq:om_expan} to $x=1$ follows from Littlewood's Tauberian theorem since
\[
 m_{2k}^{(p)}\leq\|\rho_{p}\|_{\infty}\int_{-1}^{1}t^{2k}\dd t=\mathcal{O}\left(\frac{1}{k}\right), \;\mbox{ as } k\to\infty,
\]
where $\|\rho_{p}\|_{\infty}:=\sup\{\rho_{p}(t)\mid t\in[0,1]\}$ is finite because $\rho_{p}$ is continuous on $[0,1]$.
\qed

\section{A combinatorial formula for the moments}\label{sec:comb-form}

\begin{prop}\label{prop:comb_formula}
For any $p>1$ and $k\in\N_{0}$, we have
\begin{equation}
 m_{2k}^{(p)}=2\left(\frac{p-1}{p}\right)^{p-1}\,\sum_{n=1}^{2k+1}\binom{p-1}{n}\!
 \sum_{\substack{r_{1},\dots,r_{n}\geq1 \\ r_{1}+\dots+r_{n}=2k+1}}\frac{\left(1/p\right)_{r_1}}{(r_1+1)!}\dots \frac{\left(1/p\right)_{r_n}}{(r_n+1)!},
\label{eq:comb_formula}
\end{equation}
where $(\alpha)_{n}:=\alpha(\alpha+1)\dots(\alpha+n-1)$ is the Pochhamer symbol.
\end{prop}

\begin{proof}
It follows from~\eqref{eq:int-repre} and~\eqref{eq:om_expan} that, for $|z|>1$, we have
\begin{equation}
f_{p}(z)=-\sum_{k=0}^{\infty}\frac{m_{2k}^{(p)}}{z^{2k+1}}.
\label{eq:f_p_power_series}
\end{equation}
We prove the formula~\eqref{eq:comb_formula} by expanding the analytic function $f_{p}$ in a neighborhood of $\infty$. First, we expand each of the two terms from the defining expression~\eqref{eq:f_p_alt_expr}. For the first term, assuming $|z|$ sufficiently large, we get
\begin{align*}
\left(z\left(1+\frac{1}{z}\right)^{1/q}-z\right)^{p-1}&=\left(\sum_{m=0}^{\infty}\binom{1/q}{m+1}\frac{1}{z^{m}}\right)^{p-1}\\
&=\frac{1}{q^{p-1}}\left(1+\sum_{m=1}^{\infty}\frac{(1-1/q)_{m}}{(m+1)!}\frac{(-1)^{m}}{z^{m}}\right)^{p-1}\\
&=\frac{1}{q^{p-1}}\sum_{n=0}^{\infty}\binom{p-1}{n}\left(\sum_{m=1}^{\infty}\frac{(1/p)_{m}}{(m+1)!}\frac{(-1)^{m}}{z^{m}}\right)^{n}\\
&=\frac{1}{q^{p-1}}\left(1+\sum_{n=1}^{\infty}\binom{p-1}{n}\sum_{k=1}^{\infty}\Gamma_{k}^{(n)}\,\frac{(-1)^{k}}{z^{k}}\right),
\end{align*}
where 
\[
\Gamma_{k}^{(n)}:=\sum_{\substack{r_{1},\dots,r_{n}\geq1 \\ r_{1}+\dots+r_{n}=k}}\frac{\left(1/p\right)_{r_1}}{(r_1+1)!}\dots \frac{\left(1/p\right)_{r_n}}{(r_n+1)!}.
\]
Analogously, we deduce
\[
\left(z-z\left(1-\frac{1}{z}\right)^{1/q}\right)^{p-1}=
\frac{1}{q^{p-1}}\left(1+\sum_{n=1}^{\infty}\binom{p-1}{n}\sum_{k=1}^{\infty}\Gamma_{k}^{(n)}\,\frac{1}{z^{k}}\right).
\]
Notice that $\Gamma_{s}^{(n)}=0$ if $n>s$. Then, by~\eqref{eq:f_p_alt_expr}, we have
\[
f_{p}(z)=\frac{1}{q^{p-1}}\sum_{n=1}^{\infty}\binom{p-1}{n}\sum_{s=n}^{\infty}\Gamma_{s}^{(n)}\,\frac{(-1)^{s}-1}{z^{s}}
=-\frac{2}{q^{p-1}}\sum_{k=0}^{\infty}\frac{1}{z^{2k+1}}\sum_{n=1}^{2k+1}\binom{p-1}{n}\Gamma_{2k+1}^{(n)}
\]
for all $z\in\C$ with $|z|$ sufficiently large. Comparing the coefficients by the same powers of $z$ on both sides in~\eqref{eq:f_p_power_series} and recalling that $1/p+1/q=1$, we arrive at the statement.
\end{proof}

\begin{rem}
 Formulas \eqref{eq:f_p_power_series} and \eqref{eq:comb_formula} generalize~\eqref{eq:f_p_z_large} to a complete asymptotic expansion.
\end{rem}

\begin{rem}\label{rem:mom_p_int}
If $p>1$ is an integer, formula~\eqref{eq:comb_formula} simplifies to
 \[
         m_{2k}^{(p)}=2\sum_{j=0}^{p-1} (-1)^{j+p} \binom{p-1}{j}\binom{j(p-1)/p}{2k+p},
 \]
which is a consequence of the expansion~\eqref{eq:LHS-expansion}.
\end{rem}

\begin{rem}
A straightforward computation based on~\eqref{eq:comb_formula} gives the first three even moments:
\begin{align*}
m_{0}^{(p)}&=\left(\frac{p-1}{p}\right)^{p},\\
m_{2}^{(p)}&=\left(\frac{p-1}{p}\right)^{p}\frac{3p-1}{8p},\\
m_{4}^{(p)}&=\left(\frac{p-1}{p}\right)^{p}\frac{(5p-1)(43p^{2}+p-6)}{1152p^{3}}.
\end{align*}
\end{rem}

\subsection*{Acknowledgment}
F.~{\v S}. acknowledges the support of the EXPRO grant No.~20-17749X of the Czech Science Foundation.

\bibliographystyle{acm}

\end{document}